\documentclass{article}
\usepackage{amsmath,amsthm,amssymb,amsfonts}
  \usepackage{paralist}
  \usepackage{graphics} 
  \usepackage{epsfig} 
\usepackage{graphicx}  \usepackage{epstopdf}
 \usepackage[colorlinks=true]{hyperref}
\hypersetup{urlcolor=blue, citecolor=red}

\newtheorem{theorem}{Theorem}[section]
\newtheorem{corollary}{Corollary}[theorem]
\newtheorem{lemma}[theorem]{Lemma}

\newtheorem{remark}{Remark}[section]

\title{A consistent kinetic model for a two-component mixture with an application to plasma}

\author{Christian Klingenberg, Marlies Pirner and Gabriella Puppo}
\date{}

\begin{document}
\maketitle
%
%

\bigskip


\begin{abstract}
We consider a non reactive multi component gas mixture.We propose a class of models, which can be easily generalized to multiple species. The two species mixture is modelled by a system of kinetic BGK equations featuring two interaction terms to account for momentum and energy transfer between the species.
We prove consistency of our model: conservation properties, positivity of the solutions for the space homogeneous case, positivity of all temperatures,  H-theorem and convergence to a global equilibrium in the space homogeneous case in the form of a global Maxwell distribution. Thus, we are able to derive the usual macroscopic conservation laws. In particular, by considering a mixture composed of ions and electrons, we derive the macroscopic equations of ideal MHD from our model.
\\ \\ \textbf{keywords:} multi-fluid mixture, kinetic model, BGK approximation, plasma flow.
\end{abstract}

\section{Introduction}
 
 In this paper we shall concern ourselves with a kinetic description of gases. This is traditionally done via the Boltzmann equation for the density distributions $f_1$ and $f_2$. Under certain assumptions the complicated interaction terms of the Boltzmann equation can be simplified by a so called BGK approximation, consisting of a collision frequency multiplied by the deviation of the equilibrium distribution from $f_1$ respective $f_2$. This approximation should be constructed in a way such that it  has the same main properties of the Boltzmann equation namely conservation of mass, momentum and energy, further it should have an H-theorem with its entropy inequality and the equilibrium must still be Maxwellian.
 
Here we shall focus on gas mixtures modelled via a BGK approach. In the literature one can find two types of models. Just like the Boltzmann equation for gas mixtures contains a sum of collision terms on the right-hand side, one type of BKG models also has a sum of  BGK-type interaction terms in the relaxation operator. Examples are the models of Gross and Krook \cite{gross_krook1956}, Hamel \cite{hamel1965},  Asinari \cite{asinari}, Garzo Santos Brey \cite{Garzo1989}, Sofena \cite{Sofonea2001}, see also Cercignani \cite{Cercignani_1975}. The other type of models contains only one collision term on the right-hand side. Examples for this are Andries, Aoki and Perthame \cite{AndriesAokiPerthame2002} (giving rise to an indifferentiability principle for the species), the models in \cite{Brull_2012, Groppi} or the model by Brull \cite{Brull} with an extension leading to a correct Prandtl number in the Navier Stokes equation, adapting the ES-BGK model for mixtures. 

In this paper we are interested in the first type of models. We want to describe a gas mixture with a BGK approach containing a sum of collision terms on the right-hand side. Our interest in this kind of models comes from the fact that it is still used by engineers, chemists and physicists and in numerical applications, see for example \cite{Monteferrante, Carlos2007}. Moreover BGK and ES-BGK models give rise to efficient numerical computations, which are asymptotic preserving, that is they remain efficient even approaching the hydrodynamic regime \cite{Puppo_2007, Jin_2010,Dimarco_2014, Bennoune_2008, Dimarco, Bernard_2015, Crestetto_2012}. We are motivated by the example of a mixture of electrons and ions without chemical reactions, so there is no transfer of mass from one species to the other. The particles of the two species are so different that it is desirable to maintain their contribution separately. Here the collision frequencies differ vastly, a characteristic we want to preserve in our model. Our model is so general that the related models of Gross and Krook \cite{gross_krook1956} and Hamel \cite{hamel1965} can be derived as special cases. For our model we are able to show conservation properties, the H-theorem, positivity of solutions with positive initial data in the space homogeneous case and positivity of all temperatures.

The outline of the paper is as follows: in section 2 we will present the model for two species and prove the conservation properties and the H-theorem. We discuss the relationship of the collision frequencies in the case of plasmas and show the positivity of solutions with positive initial data in the space homogeneous case and the positivity of all temperatures. In section 3, we compare our model with other models present in the literature. First we consider related models and next we compare our model with the model of Andries, Aoki and Perthame \cite{AndriesAokiPerthame2002}. 
In section 4 we consider the asymptotic limit of a mixture given by the positive and negative particles of an ionized gas, and show that we obtain the macroscopic equations of classical MHD



\section{A two species kinetic model}
In this section we will present the model for two species and prove the conservation properties and the H-theorem. Further, we especially discuss the relationship between the collision frequencies and analyse the positivity of the temperatures.
\subsection{The general form of the model}
For simplicity in the following we consider a mixture composed of two different species. Thus, our kinetic model has two distribution functions $f_1(x,v,t)> 0$ and $f_2(x,v,t) > 0$ where $x\in \mathbb{R}^3$ and $v\in \mathbb{R}^3$ are the phase space variables and $t\geq 0$ the time. They are determined by two equations to describe their time evolution. Furthermore we consider binary interactions. So the particles of one species can interact with either themselves or with particles of the other species. In the model this is accounted for introducing two interaction terms in both equations. These considerations allow us to write formally the system of equations for the evolution of the mixture. The following structure containing a sum of the collision operator is also given in \cite{Cercignani, Cercignani_1975}. 
We describe the time evolution of the number distribution functions $f_1$ and $f_2$  by the Boltzmann equation with binary interactions  for two species of particles as in \cite{Cercignani}, chapter 6.2
\begin{align*}
\partial_t f_1 + v \cdot \nabla_x f_1 + \frac{F_1}{m_1} \cdot \nabla_v f_1 = Q_{11}(f_1,f_1)+ Q_{12}(f_1,f_2),
\\
\partial_t f_2 + v \cdot \nabla_x f_2 + \frac{F_2}{m_2} \cdot \nabla_v f_2 = Q_{22}(f_2,f_2)+ Q_{21}(f_2,f_1) ,
\end{align*}
where $ \frac{F_1}{m_1} $ respective $\frac{F_2}{m_2} $ are the acceleration of the respective species due to forces $F_k$ on particles of species $k$ with mass $m_k$ for $k=1,2$ and $Q_{kl}$, $k,l=1,2$ are the collision operators for interactions of species $k$ with species $l$. \\

 Furthermore we relate the distribution functions to  macroscopic quantities by mean-values of $f_k$
\begin{align}
\begin{split}
\int f_k(v) \begin{pmatrix}
1 \\ v \\ m_k |v-u_k|^2 \\ 
\end{pmatrix} 
dv =: \begin{pmatrix}
n_k \\ n_k u_k \\ 3 n_k T_k 
\end{pmatrix},
\end{split}
\label{macrosqu}
\end{align}
where $n_k$ is the number density, $u_k$ the mean velocity and $T_k$ the temperature which is related to the pressure $p_k$ by $p_k=n_k T_k$. Note that in this paper we shall write $T_k$ instead of $k_B T_k$, where $k_B$ is Boltzmann's constant.

\subsection{Conservation  properties of the collision operators}
A model for the evolution of a mixture  should satisfy the following conservation properties:

Conservation of mass, momentum and energy of the individual species in interaction with the species itself:

\begin{enumerate} 
\item $\int Q_{kk}(f_k,f_k) dv = 0 \quad \text{for} \quad k = 1, 2, $
\item  $\int  m_k v Q_{kk}(f_k,f_k) dv= 0 \quad \text{ for} \quad k = 1, 2, $
\item  $\int m_k |v|^2 Q_{kk}(f_k,f_k) dv= 0 \quad \text{ for} \quad  k = 1, 2. $
\end{enumerate}

Conservation of total mass, momentum and energy 
\begin{enumerate}
\item $\int Q_{kl}(f_k,f_l) dv = 0 $  for $k,l=1,2,$
\item $ \int( m_1 v Q_{12}(f_1,f_2)+ m_2 v Q_{21}(f_2,f_1)) dv = 0, $
\item $ \int (m_1 |v|^2 Q_{12}(f_1,f_2)+m_2 |v|^2 Q_{21}(f_2,f_1) )dv = 0. $
\end{enumerate}
\label{cons_prop}

\subsection{The BGK approximation}

We are interested in a BGK approximation of the interaction terms. This leads us to define equilibrium distributions not only for each species itself but also for the two interspecies equilibrium distributions. Choose the collision terms $Q_{11}, Q_{12}, Q_{21}$ and $Q_{22}$ in section 2.2 as BGK operators. Then the model can be written as:

\begin{align} \begin{split} \label{BGK}
\partial_t f_1 + \nabla_x \cdot (v f_1) + \frac{F_1}{m_1} \nabla_v f_1  &= \nu_{11} n_1 (M_1 - f_1) + \nu_{12} n_2 (M_{12}- f_1),
\\ 
\partial_t f_2 + \nabla_x \cdot (v f_2)+ \frac{F_2}{m_2} \nabla_v f_2 &=\nu_{22} n_2 (M_2 - f_2) + \nu_{21} n_1 (M_{21}- f_2), 
\end{split}
\end{align}
with the Maxwell distributions
\begin{align} 
\begin{split}
M_1(x,v,t) = \frac{n_1}{\sqrt{2 \pi \frac{T_1}{m_1}}^3} \exp({- \frac{|v-u_1|^2}{2 \frac{T_1}{m_1}}}),
\\
M_2(x,v,t) = \frac{n_2}{\sqrt{2 \pi \frac{T_2}{m_2}}^3} \exp({- \frac{|v-u_2|^2}{2 \frac{T_2}{m_2}}}),
\\
M_{12}(x,v,t) = \frac{n_{12}}{\sqrt{2 \pi \frac{T_{12}}{m_1}}^3} \exp({- \frac{|v-u_{12}|^2}{2 \frac{T_{12}}{m_1}}}),
\\
M_{21}(x,v,t) = \frac{n_{21}}{\sqrt{2 \pi \frac{T_{21}}{m_2}}^3} \exp({- \frac{|v-u_{21}|^2}{2 \frac{T_{21}}{m_2}}}),
\end{split}
\label{BGKmix}
\end{align}
where $\nu_{11}$ and $\nu_{22}$ are the collision frequencies of the particles of each species with itself, while $\nu_{12}$ and $\nu_{21}$ are related to interspecies collisions. The structure of the collision terms ensures that if one collision frequency $\nu_{kl} \rightarrow \infty$ the corresponding distribution function becomes Maxwell distribution. In addition at global equilibrium, the distribution functions become Maxwell distributions with the same velocity and temperature (see section 2.8).
The Maxwell distributions $M_1$ and $M_2$ in \eqref{BGKmix} have the same moments as $f_1$ respective $f_2$. With this choice, we guarantee the conservation of mass, momentum and energy in interactions of one species with itself (see section 2.2).
The remaining parameters $u_{12}, u_{21}, T_{12}$ and $T_{21}$ will be determined using conservation of total momentum and energy, together with some symmetry considerations.

\subsection{Relationship between the collision frequencies}
The goal of this subsection is to derive an estimate for the ratio of all the relaxation parameters $\nu_{11}, \nu_{12}, \nu_{22}$ and $\nu_{21}$ in the case of a plasma.

The parameters $\nu_{12}$ and $\nu_{21}$ are linked to the interspecies collision frequency. In plasmas, the mass ratio of the two kinds of particles is $\frac{m_2}{m_1} <<1,$ where 1 denotes ions and 2 denotes electrons. In this case a common relationship found in literature \cite{bellan2006} is
\begin{equation} \label{nu}
\nu_{12}=\frac{m_2}{m_1}\nu_{21}.
\end{equation}
A motivation for this relationship in the case of a plasma can be found in \cite{bellan2006}, chapter 1.9, which we want to mention here shortly. The collision frequency is proportional to the differential cross section and the relative velocity. For the typical velocity of ions and electrons close to equilibrium one can take the thermal velocity $v_{T_i}= (\frac{2 T_i}{m_i})^{\frac{1}{2}}$, $i=1,2$ and assume that the temperatures are of the same order, $T_1\approx T_2$. The cross sections are considered equal, because they depend on the interaction potential, which in this case is the Coulomb force, that is the same for both particles. So the only thing which remains to consider is the relative velocity. Since the mass of the ions $m_1$ is much larger than the mass of the electrons $m_2$, we get in case of $\nu_{21}$  for the relative velocity of an ion and an electron \begin{align*} \left( \frac{2 T_1}{m_1} \right)^{\frac{1}{2}}-\left( \frac{2 T_2}{m_2} \right)^{\frac{1}{2}} &\approx \left( 2 T_2 \right)^{\frac{1}{2}} \left( \left( \frac{1}{m_1} \right)^{\frac{1}{2}} - \left( \frac{1}{m_2} \right)^{\frac{1}{2}} \right)\\&=\left( 2 T_2 \right)^{\frac{1}{2}} \frac{1-\left( \frac{m1}{m_2}\right)^{\frac{1}{2}}}{m_2^{\frac{1}{2}}}\approx \left( 2 T_2 \right)^{\frac{1}{2}} \left( \frac{1}{m_2} \right)^{\frac{1}{2}} ,
\end{align*}
 which is the order of magnitude of the mean velocity of the electrons. We expect the relative velocity of two electrons to have the same order of magnitude as the thermal velocity of an electron. Since $\nu_{22}$ is proportional to the relative velocity of two electrons and we only want to compare the order of magnitudes of $\nu_{21}$ and $\nu_{22}$, we conclude that $\nu_{21}$ and $\nu_{22}$ are of the same order of magnitude, so we have $$\nu_{21} \approx \nu_{22}.$$ Now consider $\nu_{11}$. The ion thermal velocity is lower by an amount of $(\frac{m_2}{m_1})^{\frac{1}{2}}$ with respect to the electrons, since $$ \left( \frac{2 T_1}{m_1} \right)^{\frac{1}{2}} = \left( \frac{m_2}{m_1} \right)^{\frac{1}{2}} \left( \frac{2 T_1}{m_2} \right)^{\frac{1}{2}} \approx \left( \frac{m_2}{m_1} \right)^{\frac{1}{2}} \left( \frac{2 T_2}{m_2} \right)^{\frac{1}{2}}
 .$$ Therefore $$ \nu_{11} \approx (\frac{m_2}{m_1})^{\frac{1}{2}} \nu_{22}.$$ 
 For an estimate of $\nu_{12}$ and $\nu_{21}$ we consider a collision of an electron head-on  with an ion.  The velocities after a collision of an ion with an electron are given by
 $$ v'_1 = v_1 - \frac{2m_2}{m_1+m_2}[(v_1-v_2)\cdot \omega] \omega,$$
 $$ v'_2 = v_2 - \frac{2m_1}{m_1+m_2}[(v_2-v_1)\cdot \omega] \omega, $$
 where $\omega$ is a unit vector in $S^2 $.  Since we consider a head-on collision this simplifies to
 $$v'_1 = v_1 - \frac{2m_2}{m_1+m_2}(v_1-v_2),$$
 $$ v'_2 = v_2 - \frac{2m_1}{m_1+m_2}(v_2-v_1) .$$
 Since $m_2$ is small compared to $m_1$, we get
 $$ v'_1 = v_1 + O(\frac{m_2}{m_1}),$$
 $$ v'_2 = v_2 +  O(1),$$
 which reflects the physical fact that collisions of a heavy particle with a light one have a bigger influence on the lighter one than on the heavy one. 
 Hence $\nu_{12}= \frac{m_2}{m_1} \nu_{22}$. \\
 To summarize, in the case of ions and electrons, the collision frequencies can be ordered as follows:
 $$ \nu_{21} \approx \nu_{22} \approx (\frac{m_1}{m_2})^{\frac{1}{2}} \nu_{11} \approx (\frac{m_1}{m_2}) \nu_{12}.$$

See also \cite{struchtrup}.
To be flexible in choosing the relationship between the collision frequencies, we now assume the relationship
\begin{equation} 
\nu_{12}=\varepsilon \nu_{21}, \quad 0 < \varepsilon \leq 1 .
\label{coll}
\end{equation}
If $\varepsilon >1$, exchange the notation $1$ and $2$ and choose $\frac{1}{\varepsilon}.$
\subsection{Conservation properties}
This section shows how the macroscopic quantities in the interspecies Maxwellians have to be chosen in order to ensure the macroscopic conservation properties.
\begin{theorem}[Conservation of the number of each species]
Assume that \begin{align} n_{12}=n_1 \quad \text{and} \quad n_{21}=n_2,  
\label{density}
\end{align} 
then 
$$
\int Q_{11}(f_1,f_1) dv =  \int Q_{12}(f_1,f_2) dv = \int Q_{22}(f_2, f_2) dv = \int Q_{21}(f_2,f_1) dv = 0 .
$$
\end{theorem}
\begin{proof}
Conservation of mass implies that in the homogeneous case $\partial_t \int f_1 dv = 0$. Therefore
$$ \nu_{11} n_1 \int (M_1 - f_1) dv + \nu_{12} n_2 \int ( M_{12} -f_1) dv =0.$$ Since $\int (M_1 - f_1)dv =0$, this equation holds provided that $n_{12}=n_1$. Similarly for the second equation, $n_{21}=n_2$.
\end{proof}
\begin{theorem}[Conservation of total momentum]
Assume the relationships \eqref{coll} and \eqref{density} hold and assume further that $u_{12}$ is a linear combination of $u_1$ and $u_2$
 \begin{align}
u_{12}= \delta u_1 + (1- \delta) u_2, \quad \delta \in \mathbb{R}.
\label{convexvel}
\end{align} Then we have conservation of total momentum
$$
\int m_1 v [Q_{11}(f_1,f_1)+Q_{12}(f_1,f_2)] dv +
\int m_2 v [Q_{22}(f_2,f_2)+Q_{21}(f_2,f_1)] dv = 0, 
$$
provided that
\begin{align}
u_{21}=u_2 - \frac{m_1}{m_2} \varepsilon (1- \delta ) (u_2 - u_1).
\label{veloc}
\end{align}
\end{theorem}
\begin{proof}
The flux of momentum of species $1$ is given by
\begin{align}
\begin{split}
f_{m_{1,2}}:= m_1 \int v \nu_{11} n_1 (M_1 - f_1) dv + m_1 \int v \nu_{12}  n_2 (M_{12}- f_1) dv \\= m_1 \nu_{12} n_1 n_2 (u_{12} - u_1)
=
m_1 \nu_{12} n_1 n_2 (1 - \delta) (u_2 - u_1). 
\end{split}
\label{flux_mom_12}
\end{align}
The flux of momentum of species $2$ is given by
\begin{align}
f_{m_{2,1}}
=m_2 \nu_{21} n_2 n_1 (u_{21} - u_2).
\label{flu_mom_21}
\end{align}
In order to get conservation of momentum we therefore need
\begin{align*}
m_1 \nu_{12} n_1 n_2 (1 - \delta) (u_2 - u_1) + m_2 \nu_{21} n_2 n_1 (u_{21} - u_2) = 0 ,
\end{align*}
which holds provided $u_{21}$ satisfies
 \eqref{veloc} under the assumption that $\nu_{12}$ and $\nu_{21}$ satisfy \eqref{coll}.
\end{proof}
\begin{remark}
If we write $\tilde{\varepsilon} = \frac{m_1}{m_2} \varepsilon$ and $\tilde{\delta}= 1- \tilde{\varepsilon}(1-\delta)$ we obtain a similar structure for $u_{21}$ as for $u_{12}$
$$ u_{21} = \tilde{\delta} u_2 + (1- \tilde{\delta}) u_1.$$
\end{remark}
\begin{theorem}[Conservation of total energy]
Assume \eqref{coll}, conditions \eqref{density}, \eqref{convexvel} and \eqref{veloc} and assume that $T_{12}$ is of the following form
\begin{align}
\begin{split}
T_{12}=  \alpha T_1 + ( 1 - \alpha) T_2 + \gamma |u_1 - u_2 | ^2,  \quad 0 \leq \alpha \leq 1, \gamma \geq 0.
\label{contemp}
\end{split}
\end{align}
Then we have conservation of total energy
$$
\int \frac{m_1}{2} |v|^2 (Q_{11}(f_1,f_1)+Q_{12}(f_1,f_2)) dv +
\int \frac{m_2}{2} |v|^2 (Q_{22}(f_2,f_2)+Q_{21}(f_2,f_1)) dv = 0,
$$
\end{theorem}
provided that
\begin{align}
\begin{split}
T_{21}=\left[ \frac{1}{3} \varepsilon m_1 (1- \delta) \left( \frac{m_1}{m_2} \varepsilon ( \delta - 1) + \delta +1 \right) - \varepsilon \gamma \right] |u_1 - u_2|^2 \\+ \varepsilon ( 1 - \alpha ) T_1 + ( 1- \varepsilon ( 1 - \alpha)) T_2.
\label{temp}
\end{split}
\end{align}
\begin{proof}
Using the energy flux of species $1$ 
$$
F_{E_1,2}:= \int \frac{m_1}{2} |v|^2 \nu_{11} n_1 (M_1 - f_1) dv + \int \frac{m_1}{2} |v|^2 \nu_{12} n_2 (M_{12}- f_1) dv 
$$$$= \varepsilon \nu_{21} \frac{1}{2} n_2 n_1 m_1( |u_{12}|^2 - |u_1|^2)+ \frac{3}{2} \varepsilon \nu_{21} n_1 n_2  (T_{12} - T_1),  
$$
where we used \eqref{convexvel} and \eqref{contemp}. Analogously the energy flux of species $2$ towards $1$ is 
$$
F_{E_{2,1}}= \nu_{21} m_2 n_1 n_2 ( |u_{21}|^2 - |u_2|^2)+ \frac{3}{2} \nu_{21} n_1 n_2 ( T_{21} -T_2). 
$$
Substitute $u_{21}$ with \eqref{veloc} and $T_{21}$ from \eqref{temp}. This permits to rewrite the energy fluxes as 
\begin{align}
\begin{split}
F_{E_{1,2}}= \varepsilon \nu_{21} \frac{1}{2} n_2 n_1 m_1\left[(\delta^2-1) |u_1|^2 + (1- \delta)^2 |u_2|^2 + 2 \delta ( 1- \delta) u_1 \cdot u_2 \right]\\ + \frac{3}{2} \varepsilon \nu_{21} n_1 n_2  \left[(1-\alpha) (T_2-T_1) + \gamma |u_1-u_2|^2\right],
\end{split}
 \label{flux_en_12}
\end{align}
\begin{align}
\begin{split}
F_{E_{2,1}} =\frac{1}{2} \nu_{21} m_2 n_1 n_2  \big[ \left( (1-\frac{m_1}{m_2} \varepsilon (1- \delta) )^2 -1 \right) |u_2|^2 
+ \left( \frac{m_1}{m_2} \varepsilon (\delta - 1) \right)^2 |u_1|^2  \\
+ 2 ( 1- \frac{m_1}{m_2} \varepsilon (1-\delta)) \frac{m_1}{m_2} \varepsilon ( 1- \delta) u_1 \cdot u_2 
\big] 
+ \frac{3}{2} \nu_{21} n_1 n_2 \big[ \varepsilon ( 1- \alpha) (T_1-T_2)\\ + \left( \frac{1}{3} \varepsilon m_1 (1- \delta) \left( \frac{m_1}{m_2} \varepsilon ( \delta - 1) + \delta +1 \right) - \varepsilon \gamma \right) |u_1 - u_2|^2 \big].
\end{split}
\label{flux_en_21}
\end{align}
Adding these two terms, we see that the total energy is conserved.
\end{proof}
\begin{remark}
We have $0 \leq 1-\varepsilon (1 - \alpha) \leq 1$ and  $0 \leq \varepsilon (1- \alpha) \leq 1$, so that in \eqref{temp} the two terms with the temperatures are also a convex combination of $T_1$ and $T_2$.  
\end{remark}
\begin{remark}
The remaining free parameters can be fixed for specific situations. For example, if we see the parameters $\alpha, \delta, \gamma$ and $\varepsilon$ from the model presented in this paper as functions of the masses $m_1$ and $m_2$, we can get more restrictions on these parameters by physical considerations. 
\begin{itemize}
\item In the limit $\frac{m_1}{m_1+m_2}\rightarrow 0$, we expect that $u_{12}=u_2$ and $T_{12}=T_2$, since we expect that light particles are driven by the flow of the heavy particles, so they adapt the velocity and the fluctuations to the mean velocity of the heavy particles. If we look at \eqref{convexvel}, \eqref{veloc}, \eqref{contemp} and \eqref{temp}, the definitions of $u_{12}, u_{21}, T_{12}$ and $T_{21}$, we see in order to realize this, we need $\delta \rightarrow 0, \alpha \rightarrow 0$ and $\gamma \rightarrow 0$.
\item In the limit $\frac{m_1}{m_1+m_2} \rightarrow \frac{1}{2}$, when the mass of the particles become indistinguishable, we expect $T_{12}=T_{21}$ and $u_{12}=u_{21}$. For this we need $\delta \rightarrow  \frac{\varepsilon}{1+ \varepsilon}$, $\alpha  \rightarrow \frac{\varepsilon}{1+ \varepsilon}$ and $\gamma \rightarrow \frac{1}{3} m_1 \frac{\varepsilon}{(1-\varepsilon)^2}$. 
\item In the limit $\frac{m_1}{m_1+m_2}\rightarrow 1$, the heavy particles don't feel the other particles, so we expect that we have no change in the mean velocity and in the temperature, e.g $u_{12}=u_1$ and $T_{12}=T_1$. Here we need $\delta \rightarrow 1, \alpha \rightarrow 1$ and $\gamma \rightarrow 0$.
\end{itemize}
\label{consider}
\end{remark}
\subsection{Positivity of the distribution function}
We want to show that in the space homogeneous case positive initial values of the distribution functions stay non-negative when their time evolution is described by the two species BGK model described in this paper. 

\begin{theorem}[Non-negative solutions of the BGK equation for two species]
Assume $f_1(\cdot,t), f_2(\cdot,t) \in C^1( \mathbb{R}^3)$ and $f_1(v,\cdot),$ $ f_2(v, \cdot)\in C^1(\mathbb{R}^+_0)$, $\nu_{11}(t),$ $ \nu_{12}(t),$ $ \nu_{21}(t),$ $ \nu_{22}(t)\geq 0$ and we have no external forces. If $f_1(v,0), f_2(v,0)\geq 0$, we have $f_1(v,t), f_2(v,t)> 0$ for every $t\geq 0$.
\end{theorem}
\begin{proof}
In the space-homogeneous case we get from the conservation properties that $n_1$ and $n_2$ are constant in time.
Rewrite \eqref{BGK} in the space homogeneous case as
\begin{align}
\begin{split}
\partial_t f_1(v,t) +  \nu_{11}(t) n_1 f_1(v,t) + \nu_{12}(t) n_2  f_1(v,t) \\= \nu_{11}(t) n_1 M_1(v,t) + \nu_{12}(t) n_2 M_{12}(v,t) .
\label{equation}
\end{split}
\end{align}
Define $g_1(v,t)= f_1(v,t) e^{\alpha(t)}$ for some differentiable function $\alpha(t)$ determined later. Then 
\begin{align*}
\partial_t g_1(v,t) &= \partial_t f_1(v,t) e^{\alpha (t)} + f_1(v,t) e^{\alpha(t)} \partial_t \alpha(t)\\&= \partial_t f_1(v,t) e^{\alpha (t)} + g_1(v,t) \partial_t \alpha(t).
\end{align*}

By using \eqref{equation} we get
\begin{align*}
\partial_t g_1(v,t) =&-(\nu_{11}(t) n_1 + \nu_{12}(t) n_2)  g_1(v,t) \\&+ (\nu_{11}(t) n_1 M_1(v,t)+ \nu_{12}(t) n_2 M_{12}(v,t) ) e^{\alpha(t)}+ g_1(v,t) \partial_t\alpha(t).
\end{align*}
Now choose $\alpha(t)$ such that
$$ \partial_t \alpha(t) = \nu_{11}(t) n_1 + \nu_{12}(t) n_2 \quad \text{and} \quad \alpha(0)=0,$$ so we choose $\alpha(t)$ as 
$$ \alpha(t) = \int_0^t \nu_{11}(s) n_1 + \nu_{12}(s) n_2 ds.$$
The initial value of $\alpha$ is chosen such that $f_1$ and $g_1$ have the same initial values.
Then $g_1$ solves
\begin{align*}
\partial_t g_1(v,t) = (\nu_{11}(t) n_1 M_1(v,t) + \nu_{12}(t) n_2 M_{12}(v,t) ) e^{\alpha(t)},
\end{align*}
or in integral form
{\small
\begin{align*}
g_1(v,t) &= g_{1}(v,0)+ \int_0^t [\nu_{11}(s) n_1 M_1(v,t) + \nu_{12}(s) n_2 M_{12}(v,t)] e^{\alpha(s)} ds,
\end{align*}}
so
{\small
\begin{align*}
e^{\alpha(t)} &f_1(v,t) = f_{1}(v,0) + \int_0^t [\nu_{11}(s) n_1 M_1(v,t) + \nu_{12}(s) n_2 M_{12}(v,t) ] e^{\alpha(s)} ds .
\end{align*}}
 Since we assumed $\nu_{11}(t), \nu_{12}(t), \nu_{21}(t), \nu_{22}(t) \geq 0$ 
for every $ t \leq t_0$ and positive initial values, all terms on the right-hand side are positive. Hence $f_1$ is positive. \\
Similar for $f_2$.
\end{proof}
\subsection{Positivity of the temperatures}
\begin{theorem}
Assume that $f_1(x,v,t), f_2(x,v,t) > 0$. Then all temperatures $T_1$, $T_2$, $T_{12}$ given by \eqref{contemp} and $T_{21}$ given by \eqref{temp} are positive provided that 
 \begin{align}
0 \leq \gamma \leq \frac{m_1}{3} (1-\delta) \left[(1 + \frac{m_1}{m_2} \varepsilon ) \delta + 1 - \frac{m_1}{m_2} \varepsilon \right] .
 \label{gamma}
 \end{align}
\end{theorem}
\begin{proof}
$T_1$ and $T_2$ are positive as integrals of positive functions. $T_{12}$ is positive because by construction it is a convex combination of $T_1$ and $T_2$. For $T_{21}$ we consider the coefficients in front of $|u_1-u_2|^2$, $T_1$ and $T_2$. The term in front of $T_1$ is positive by definition. The positivity of the term in front of $T_2$ is equivalent to the condition $\alpha \geq 1- \frac{1}{\varepsilon}$, which is satisfied since $\varepsilon \leq1$, the positivity of the term in front of $|u_1-u_2|^2$ is equivalent to the condition \eqref{gamma}.
\end{proof}
\begin{remark}
According to the definition of $\gamma$, $\gamma$ is a non-negative number, so the right-hand side of the inequality in \eqref{gamma} must be non-negative. This condition is equivalent to 
\begin{align}
 \frac{ \frac{m_1}{m_2}\varepsilon - 1}{1+\frac{m_1}{m_2}\varepsilon} \leq  \delta \leq 1 .
\label{gammapos}
\end{align}
If the collision frequencies are linked as in \eqref{coll}, $\varepsilon = \frac{m_2}{m_1}$, then the right-hand side of \eqref{gamma} is always positive.
\end{remark}
\subsection{H-theorem for mixtures}
\begin{remark}
From the case of one species BGK model we know that
$$ \int M_k \ln M_k dv \leq \int f_k \ln f_k dv,$$ for $k=1,2$, see for example problem 1.7.1 in \cite{Cercignani}.
\label{one}
\end{remark}
\begin{lemma}
Assuming \eqref{contemp} and \eqref{temp} and the positivity of the temperatures  \eqref{gamma}, we have the following inequality
\begin{align}
\varepsilon \ln T_{12} + \ln T_{21}  \geq \varepsilon \ln T_1 + \ln T_2.
\label{lemma}
\end{align}
\label{inequ}
\end{lemma}
\begin{proof}
We start with the left-hand side of \eqref{inequ}. First we insert the definition of $T_{12}$ and $T_{21}$ from  \eqref{contemp} and \eqref{temp}. Since $\gamma$ and the term in front of $|u_1-u_2|^2$ in \eqref{temp} are positive, we can use  the monotonicity of the logarithm and get
\begin{align*}
\varepsilon &\ln T_{12} + \ln T_{21}\\&= 
\varepsilon \ln \left[ \alpha T_1 + (1- \alpha) T_2 + \gamma |u_1 - u_2|^2 \right] \\ &+ \ln \big[\frac{1}{3} \varepsilon m_1 (1- \delta) ( \frac{m_1}{m_2} \varepsilon ( \delta - 1) + \delta +1) - \varepsilon \gamma ) |u_1 - u_2|^2 \\&+ \varepsilon ( 1 - \alpha ) T_1 + ( 1- \varepsilon ( 1 - \alpha)) T_2 \big] \\ &\geq \varepsilon \ln( \alpha T_1 + (1- \alpha) T_2) + \ln (\varepsilon ( 1 - \alpha ) T_1 + ( 1- \varepsilon ( 1 - \alpha)) T_2) .
\end{align*}
If we now use the  concavity of the logarithm and the assumptions $0 \leq \alpha \leq 1, \varepsilon <1$, the expression above can be bounded from below by
\begin{align*}
\varepsilon \alpha \ln T_1 + \varepsilon (1-\alpha) \ln T_1 + (1- \varepsilon (1- \alpha)) \ln T_2 +\varepsilon(1- \alpha) \ln T_2,
\end{align*}
which gives the inequality stated in lemma \ref{inequ}.
\end{proof}
\begin{theorem}[H-theorem for mixture]
Assume $f_1, f_2 >0$.
Assume the relationship between the collision frequencies \eqref{coll} , the conditions for the interspecies Maxwellians \eqref{density} , \eqref{convexvel}, \eqref{veloc}, \eqref{contemp} and \eqref{temp} with $\alpha, \delta \neq 1$ and the positivity of the temperatures \eqref{gamma}, then
\begin{align*}
\int (\ln f_1) ~ Q_{11}(f_1,f_1) &+ (\ln f_1) ~ Q_{12}(f_1,f_2) dv \\&+ \int (\ln f_2) ~ Q_{22}(f_2, f_2)+ (\ln f_2) ~ Q_{21}(f_2, f_1) dv\leq 0 ,
\end{align*}
with equality if and only if $f_1$ and $f_2$ are Maxwell distributions with equal velocity and temperature. 
\label{H-theorem}
\end{theorem}
\begin{proof}
The fact that $\int \ln f_k Q(f_k,f_k) \leq 0$, $k=1,2$ is shown in proofs of the H-theorem of the single BGK-model, for example in \cite{struchtrup}.
In both cases we have equality if and only if $f_1=M_1$ and $f_2 = M_2$. \\
Let us define 
$$
S(f_1,f_2) :=\nu_{12} n_2 \int \ln f_1  ( M_{12}- f_1 )dv + \nu_{21} n_1 \int \ln f_2  ( M_{21}- f_2)dv .
$$
The task is to prove that $S(f_1, f_2)\leq 0$.
Since the function $H(x)= x \ln x -x$ is strictly convex for $x>0$, we have $H'(f) (g-f) \leq H(g) - H(f)$ with equality if and only if $g=f$. So \begin{align}
(g-f) \ln f  \leq g \ln g - f \ln f +f -g .
\label{convex}
\end{align}
Consider now $S(f_1,f_2)$ and apply the inequality \eqref{convex} to each of the two terms in $S$.
$$
S \leq
\nu_{12} n_2 \left[ \int M_{12} \ln M_{12}  dv - \int f_1 \ln f_1 dv - \int M_{12} dv + \int f_1 dv\right]$$ $$ + \nu_{21} n_1 \left[\int M_{21} \ln M_{21} dv - \int f_2 \ln f_2 dv - \int M_{21} dv + \int f_2 dv \right], 
$$
with equality if and only if $f_1=M_{12}$ and $f_2=M_{21}$. Then $u_{12}= \delta u_1 + (1- \delta) u_2 = u_1$ from which we can deduce $u_1=u_2=u_{21}=u_{12}$ and $T_1=T_2=T_{12}=T_{21}$. This means $f_1$ and $f_2$ are Maxwell distributions with equal bulk velocity and temperature. \\
Since $M_{12}$ and $f_1$ have the same density and $M_{21}$ and $f_2$ have the same density, too,  the right-hand side reduces to
$$
\nu_{12} n_2 ( \int M_{12} \ln M_{12} dv - \int f_1 \ln f_1 dv )+ \nu_{21} n_1 (\int M_{21} \ln M_{21} dv - \int f_2 \ln f_2 dv ) .
$$
Since $\int M \ln M dv = n \ln(\frac{n}{\sqrt{\frac{2 \pi T}{m}}^3})- \frac{3}{2} n$ for $M=\frac{n}{\sqrt{\frac{2 \pi T}{m}}^3}e^{-\frac{|v-u|^2}{\frac{2T}{m}}},$ we will have that
$$
\nu_{12} n_2  \int M_{12} \ln M_{12}  dv + \nu_{21} n_1 \int M_{21} \ln M_{21}  dv $$$$\leq \nu_{21} n_1 \int M_2 \ln M_2  dv  + \nu_{12} n_2  \int M_1 \ln M_1  dv, 
$$
provided that
\begin{align*}
\nu_{12} n_2 n_1 \ln \frac{n_1}{\sqrt{2 \pi \frac{T_{12}}{m_1}}^3} +\nu_{21} n_2 n_1 \ln \frac{n_2}{\sqrt{2 \pi \frac{T_{21}}{m_2}}^3} \\ \leq \nu_{12} n_2 n_1 \ln \frac{n_1}{\sqrt{2 \pi \frac{T_1}{m_1}}^3} +\nu_{21} n_2 n_1 \ln \frac{n_2}{\sqrt{2 \pi \frac{T_2}{m_2}}^3},
\end{align*}
which is equivalent to the condition \eqref{lemma} proven in Lemma \ref{inequ}.
\\
With this inequality we get
\begin{align*}
S(f_1,f_2) \leq
&\nu_{12} n_2 [ \int M_1 \ln M_{1}  dv - \int f_1 \ln f_1 dv ]\\&+ \nu_{21} n_1 [ M_2 \ln M_{2} dv - \int f_2 \ln f_2 dv ] \leq 0 .
\end{align*}
The last inequality follows from remark (\ref{one}). Here we also have equality if and only if $f_1=M_1$ and $f_2=M_2$, but since we already noticed that equality also implies $f_1=M_{12}$ and $f_2=M_{21}$, we also have $T_{21}=T_2=T_1=T_{12}$ and $u_1=u_2=u_{12}=u_{21}$.

\end{proof}
Define the total entropy $H(f_1,f_2) = \int (f_1 \ln f_1 + f_2 \ln f_2) dv$. We can compute 
$$ \partial_t H(f_1,f_2) + \nabla_x \cdot \int ( f_1 \ln f_1 + f_2 \ln f_2 ) v dv + \int F_i \nabla_v ( f_1 \ln f_1 + f_2 \ln f_2 ) dv = S(f_1,f_2),$$ by multiplying the BGK equation for the species $1$ by $\ln f_1$, the BGK equation for the species $2$ by $\ln f_2$ and integrating the sum with respect to $v$.
 \begin{corollary}[Entropy inequality for mixtures]
Assume $f_1, f_2 >0$. Assume $\nabla_v \cdot F_i =0$ and a fast enough decay of $f$ to zero for $v\rightarrow \infty$.
Assume relationship \eqref{coll}, the conditions \eqref{density} , \eqref{convexvel}, \eqref{veloc}, \eqref{contemp} and \eqref{temp} with $\alpha, \delta \neq 1$ and the positivity of the temperatures \eqref{gamma} , then we have the following entropy inequality
$$
\partial_t \left(\int   f_1 \ln f_1  dv + \int f_2 \ln f_2 dv \right) + \nabla_x \cdot \left(\int  v f_1 \ln f_1  dv + \int v f_2 \ln f_2 dv \right) \leq 0,
$$
with equality if and only if $f_1$ and $f_2$ are Maxwell distributions with equal bulk velocity and temperature. Moreover at equilibrium the interspecies Maxwellians $M_{12}$ and $M_{21}$ satisfy $u_{12} = u_2=u_1= u_{21}$ and $T_{12} = T_2=T_1=T_{21}$.
\end{corollary}
We now explicitly specify the global equilibrium.

\subsection{The structure of the equilibrium}
\begin{theorem}[Equilibrium]
Assume $f_1, f_2 >0$.
Assume relationship \eqref{coll} , the conditions \eqref{density}, \eqref{convexvel}, \eqref{veloc}, \eqref{contemp} and \eqref{temp}and the positivity of the temperatures \eqref{gamma}.
Then $Q_{11}(f_1,f_1)+Q_{12}(f_1,f_2)=0$ and $Q_{22}(f_2,f_2)+Q_{21}(f_2,f_1)=0$, if and only if $f_1$ and $f_2$ are Maxwell distributions with equal mean velocity and temperature.
\end{theorem}
\begin{proof}
If $Q_{11}(f_1,f_1)+Q_{12}(f_1,f_2)=0$ and $Q_{22}(f_2,f_2)+Q_{21}(f_2,f_1)=0$ , then $\ln f_1 ~ Q_{11}(f_1,f_1)+\ln f_1 ~ Q_{12}(f_1,f_2)+\ln f_2 ~ Q_{22}(f_2,f_2)+\ln f_2 ~ Q_{21}(f_2,f_1)=0$ and so we have equality in the H-theorem.
\end{proof}
\begin{remark}
 Once the species have an equal mean velocity and temperature the indifferentiability principle in \cite{AndriesAokiPerthame2002} holds. Thus in our model two identical species are indifferentiable once they reach equilibrium. For more information about the indifferentiability principle see section 3 or \cite{AndriesAokiPerthame2002}.
\end{remark}
\section{Special cases of this model in the literature}
In this section, we review models that have been previously introduced, \cite{gross_krook1956} and \cite{hamel1965}, and that can be considered as special cases of the class described here. Thanks to this, all of them enjoy an H-theorem, conservation properties and positivity of the interspecies temperatures.
\subsection{Model of Gross and Krook}
The model of Gross and Krook \cite{gross_krook1956} is obtained by choosing $\varepsilon = 1$, while $\delta$, $\alpha$ and $\gamma$ are free parameters. In the case of a plasma they suggest $\delta=\frac{m_1}{m_1+m_2}$.
They also assume \eqref{density} for conservation of mass. They assume one of the mixture velocities to be a linear combination of $u_1$ and $u_2$, similar to \eqref{convexvel} and deduce \eqref{veloc} from conservation of momentum. They further choose $T_{12}$ of the form $$T_{12}= \alpha T_1 + (1- \alpha) T_2 + A |u_1|^2 + B |u_2|^2 + C u_1 u_2, \quad \alpha, A,B,C \in \mathbb{R},$$ and deduce with conservation of energy that $T_{21}$ is given by $$T_{21} = (1-\alpha) T_1 + \alpha T_2 + D |u_1|^2 + E |u_2|^2 + F u_1 u_2,$$ where five of the variables $A,B,C,D,E \in \mathbb{R}$ are determined in order to get conservation of energy. From the present work the constants must be chosen in order to satisfy \eqref{gamma}. In this case the model satisfies the H-Theorem.
\subsection{Model of Hamel}
Hamel's model \cite{hamel1965} is obtained by choosing $\varepsilon=1$, $ \delta= \frac{m_1}{m_1+m_2}$, $\alpha=\frac{m_1^2+m_2^2}{(m_1+m_2)^2}$ and $\gamma=\frac{m_1 m_2}{(m_1+m_2)^2} \frac{m_2}{3}$.
The parameters are chosen in order to reproduce the fluxes of momentum and energy of Maxwellian molecules. His model also takes into account the physical considerations described in remark \ref{consider}.
The model satisfies condition \eqref{density} for conservation of mass. $u_{12}$ and $u_{21}$ satisfy condition \eqref{convexvel} respectively \eqref{veloc} with this chosen $\delta$ and $\varepsilon$, so we have conservation of total momentum.
$T_{12}$ and $T_{21}$ are of the form \eqref{contemp} respective \eqref{temp}, so we have conservation of energy. 
The requirements for positivity of the temperature are satisfied, since $$ \alpha = \frac{m_1^2+ m_2^2}{(m_1+m_2)^2} \geq 0 = 1 - \frac{1}{\varepsilon} \quad \text{for} \quad \varepsilon =1,$$ and conditions \eqref{gamma} and \eqref{gammapos} then reduce to $m_1 \geq 0$ and $m_2 \geq 0$, so Hamel's model has positive temperatures and an H-theorem. \\
\subsection{Comparison with the model of Andries, Aoki and Perthame}
The next model also describes a gas mixture of Maxwellian molecules, but it contains only one term on the right-hand side \cite{AndriesAokiPerthame2002}.
\begin{align}
\begin{split}
\partial_t f_1 + v \cdot \nabla_x f_1 &= (\nu_{11} n_1 + \nu_{12} n_2)  (M^{(1)} - f_1),
\\
\partial_t f_2 + v \cdot \nabla_x f_2 &= (\nu_{22} n_1 + \nu_{21} n_1)  (M^{(2)} - f_2).
\end{split}
\end{align}
The Maxwell distributions are given by
\begin{align}
\begin{split}
M^{(1)}&=\frac{n_1}{\sqrt{2 \pi \frac{T^{(1)}}{m_1}}^3 } e^{- \frac{m_1|v-u^{(1)}|^2}{2 T^{(1)}}},
\\
M^{(2)}&= \frac{n_2}{\sqrt{2 \pi \frac{T^{(2)}}{m_2}}^3 } e^{- \frac{m_2|v-u^{(2)}|^2}{2 T^{(2)}}},
\end{split}
\end{align}
with the interspecies velocities
\begin{align}
\begin{split}
\\
u^{(1)}&= u_1 + 2 \frac{m_2}{m_1 + m_2} \frac{\chi_{12}}{\nu_{11} n_1 + \nu_{12} n_2} n_2 ( u_2 - u_1 ),
\\
u^{(2)}&= u_2 + 2 \frac{m_1}{m_1 + m_2} \frac{\chi_{21}}{\nu_{22} n_2 + \nu_{21} n_1} n_1 ( u_1 - u_2 ), 
\end{split}
\end{align}
and the interspecies temperatures
\begin{align}
\begin{split}
T^{(1)} &= T_1 - \frac{m_1}{3} |u^{(1)}- u_1|^2 \\ &+ \frac{2}{3} \frac{m_1 m_2 }{(m_1 + m_2)^2} \frac{4 \chi_{12}}{\nu_{11} n_1 + \nu_{12} n_2} n_2 ( \frac{3}{2}(T_2 - T_1 ) + m_2 \frac{|u_2- u_1|^2}{2}),
\\
T^{(2)} &= T_2 - \frac{m_2}{3} |u^{(2)}- u_2|^2 \\ &+ \frac{2}{3} \frac{m_1 m_2 }{(m_1 + m_2)^2} \frac{4 \chi_{21}}{\nu_{22} n_2 + \nu_{21} n_1} n_1 ( \frac{3}{2}(T_1 - T_2 ) + m_1 \frac{|u_1 - u_2 |^2}{2}),
\end{split}
\end{align}
where $\chi_{12},$ $\chi_{21},$ $\nu_{12}$ and $\nu_{21}$ are parameters which are related to the differential cross section. For the detailed expressions see \cite{AndriesAokiPerthame2002}.

The model also satisfies the conservation properties and the H-theorem with equality if and only if the distribution functions are Maxwell distributions with equal mean velocity and temperature. 
\begin{remark}
The flux of momentum of the species $1$ in this model is given by 
\begin{align*}
m_1 \nu_1 \int n_2 v (M^{(1)}-f_1) dv= 2 \frac{m_2 m_1}{m_2+m_1} \chi_{12} ( u_2 - u_1) n_1 n_2 .
\end{align*}
The flux of the energy of species $1$ is given by 
\begin{align*}
\int \frac{m_1}{2} |v|^2 \nu_1 (M^{(1)} - f_1 ) dv = n_1 n_2 \frac{2 m_2 m_1 \chi_{12}}{(m_1+m_2)}(-\frac{m_1}{m_1+m_2} |u_1|^2 + \frac{m_2}{m_1+m_2} |u_2|^2 \\+ \frac{m_1-m_2}{m_1+m_2} u_1 u_2+\frac{2}{m_1+m_2}\frac{3}{2}(T_2-T_1)).
\end{align*}
So the model discussed here reproduces the same momentum and energy fluxes between the species \eqref{flux_mom_12}, \eqref{flu_mom_21}, \eqref{flux_en_12} and \eqref{flux_en_21} choosing the parameters $\delta, \alpha$ and $\gamma$ as:
\begin{align*}
\delta &= -2 \frac{m_2}{m_1+m_2} \frac{\chi_{12}}{\nu_{12}}+1, \\
\alpha &= - 4 \frac{m_1 m_2}{(m_1 + m_2)^2} \frac{\chi_{12}}{\nu_{12}} +1, \\
\gamma &= \frac{4}{3} \frac{m_1 m_2^2}{(m_1+m_2)^2} \frac{\chi_{12}}{\nu_{12}} n_1 n_2(1- \frac{\chi_{12}}{\nu_{12}}).
\end{align*}
For $\chi_{12}\leq\nu_{12}$ the parameter $\gamma$ is non-negative. 
\end{remark}
The model of Andries, Aoki and Perthame  has another property, proposition 3.2 in \cite{AndriesAokiPerthame2002},  which the models described above do not have. It is called the indifferentiability principle.
It denotes the following property: 
\begin{remark}[Indifferentiability principle]
When the masses $m_1$ and $m_2$ and the collision frequencies $\nu_{11}, \nu_{12}, \nu_{21}$ and $\nu_{22}$ are  identical, the total distribution function $f=f_1+f_2$ obeys a single species BGK equation.
\end{remark}
See also \cite{Brull} for another model which also has the indifferentiability principle.
The model in this paper does not satisfy the indifferentiability principle. The indifferentiability principle in our model holds only in the global equilibrium. On physical grounds it is reasonable to assume that two species of identical particles become really indifferentiable when they have the same macroscopic speeds and temperatures.
\section{Deriving macroscopic MHD equations} 
In this section we want to illustrate the model in the case of ions and electrons. Finally, we derive the typical macroscopic equations for a mixture composed of ions and electrons, the equations of ideal Magnetohydrodynamics, from our model. You can also find a similar derivation in \cite{degond2004} but for an isothermal flow.
\subsection{The BGK model for ions and electrons}
We consider the case of ions and electrons and set $\varepsilon=\frac{m_2}{m_1}$ as it is motivated in section 2.4 or \cite{bellan2006}. For simplicity we take $\delta=0$, $\alpha= \frac{m_2}{m_1+m_2}$ and $\gamma=0$, although the MHD equations can also derived from the general model. We replace the index $1$ by $i$ for the ions and $2$ by $e$ for the electrons. Then the particles are subjected to the Lorentz force $F_i= e (E+ v \times B)$ and $F_e=-e (E+v \times B )$, where $e$ is the elementary charge and $E$ and $B$ the mean electric and magnetic fields given by the Maxwell equations. In this case the model \eqref{BGK} rewrites as
\begin{align} \begin{split} 
\partial_t f_i + \nabla_x \cdot (v f_i) + \frac{e (E+ v \times B)}{m_i} \nabla_v f_i  &= \nu_{ii} n_i (M_i - f_i) + \nu_{ie} n_e (M_{ie}- f_i),
\\ 
\partial_t f_e + \nabla_x \cdot (v f_e)- \frac{e (E+ v \times B)}{m_e} \nabla_v f_e &=\nu_{ee} n_e (M_e - f_e) + \frac{m_i}{m_e} \nu_{ie} n_i (M_{ei}- f_e) .
\end{split}
\label{BGKplasma}
\end{align}
\subsection{Macroscopic equations for ions and electrons}
In order to derive macroscopic equations, we multiply the first equation of (\ref{BGK}) with $(1,m_i v,\frac{m_i}{2}|v|^2)$, and the second with $(1,m_e v,\frac{m_e}{2}|v|^2)$. Then we integrate  them with respect to the velocity.The obtained macroscopic system is not closed since we obtain terms of the form $\int v \otimes v f_i dv, \int v \otimes v f_e dv, \int |v|^2 v f_i dv$ and $\int |v|^2 v f_e dv$. All the other terms are functions of known quantities given in \eqref{macrosqu}. There are plasmas where the two species first relax to its own equilibrium and then to a global one. According to Chapter 1.9 in \cite{bellan2006}, we expect that plasmas are typically not in thermodynamic equilibrium, although the components may be in a partial equilibrium. This means, the electrons are in thermal equilibrium with itself but not with the ions, and the other way round. So in our considerations we assume that each species is in equilibrium with itself, e.g. setting $f_i=M_i$ and $f_e=M_e$. In this way, we obtain a closed system of equations for the conservation of mass, momentum and energy.
\begin{align}
\partial_t n_k + \nabla_x(n_k u_k)=0, \quad k=i,e,
\end{align}
\begin{align}
\begin{split}
\partial_t(m_i n_i u_i)+\nabla_x (n_i T_i) + \nabla_x \cdot (m_i u_i \otimes u_i n_i ) - e n_i (E+u_i \times B) \\= 
\nu_{ie} m_i n_e n_i (u_e - u_i),
\end{split}
\end{align}
\begin{align}
\begin{split}
\partial_t(m_e n_e u_e)+\nabla_x (n_e T_e) + \nabla_x \cdot (m_e u_e \otimes u_e n_e ) + e n_e (E+u_e \times B) \\= 
\nu_{ei} m_e n_e n_i (u_i - u_e),
\end{split}
\end{align}
\begin{align}
\begin{split}
\partial_t(\frac{m_i}{2} n_i |u_i|^2 + \frac{3}{2} n_i T_i) + \nabla_x \cdot(\frac{5}{2} p_i u_i) +\nabla_x \cdot( \frac{m_i}{2} n_i |u_i|^2 u_i) - e n_i E u_i  \\= \frac{1}{2}\nu_{ie} n_e n_i m_i( |u_e|^2 - |u_i|^2)+ \nu_{ie}\frac{3}{2} n_i n_e  \frac{m_i}{m_e+m_i}( T_e -T_i),  
\end{split}
\end{align}
\begin{align}
\begin{split}
\partial_t(\frac{m_e}{2} n_e |u_e|^2 + \frac{3}{2} n_e T_e) + \nabla_x \cdot(\frac{5}{2} p_e u_e )+ \nabla_x \cdot(\frac{m_e}{2} n_e |u_e|^2 u_e) +e  n_e E u_e  \\=\frac{1}{2} \nu_{ei} n_e n_i m_e( |u_i|^2 - |u_e|^2)+ \nu_{ei}\frac{3}{2} n_i n_e  \frac{m_e}{m_e+m_i}( T_i -T_e). 
\end{split}
\end{align}
In order to determine the time evolution of the electric and magnetic field, we couple the system with the Maxwell equations.
\begin{align}
\nabla_x \cdot E &= \frac{1}{\varepsilon_0} \rho_c,
\\
\nabla_x \times E + \frac{\partial B}{\partial t}&=0,
\\
\nabla_x \times B &= \mu_0 j + \mu_0 \varepsilon_0 \frac{\partial E}{\partial t},
\\
\nabla_x \cdot B &=0,
\\
\rho_c &= e(n_i-n_e),
\\
j&=e(n_i u_i-n_e u_e),
\end{align}
where $c^2=\frac{1}{\mu_0 \varepsilon_0}$ is the speed of light and $\mu_0, \varepsilon_0$ the magnetic and electric vacuum permittivity. 
\subsection{Dimensionless equations}
First we define dimensionless variables of the time $t$, the length $x$, the velocities $u_{e}, u_i$, the number densities $n_{e}, n_i$, the temperatures $T_{e}, T_i$, the magnetic field $B$, the electric field $E$, the electron-ion collision frequency $\nu_{ei}$, the ion-electron collision frequency $\nu_{ie}$ and the current density $j$, for example $t'=$\mbox{$^t$/$_{\bar{t}}$} for a typical time scale $\bar{t}$. In particular, the order of magnitudes of some quantities are assumed to be linked: We assume that both species have densities, mean velocities and temperatures of the same order of magnitude, e.g. $ \bar{n}_i = \bar{n}_e = \bar{n}$, $ \quad \bar{u}_i = \bar{u}_e =\bar{u}=$ \mbox{$^{\bar{x}}$/$_{\bar{t}}$} and $\bar{T}_i=\bar{T}_e = \bar{T}$. With the last two assumptions we assume that we are close to a thermodynamic equilibrium in which the two mean velocities and temperatures would be equal. Further, we assume that $ \bar{E}=\bar{B}\bar{u}$. From non-dimensionalizing the first two Maxwell equations we see that this means that the electric field induced by a change of the magnetic field in time dominates over the fields which arise from charges and currents. Further we assume that $\bar{B}= \mu_0 \bar{x} \bar{j}$, which means that the magnetic field induced by currents dominates over the magnetic field due to changes of the electric field in time. Finally, we assume that $\bar{\nu}_{ie}=\frac{m_e}{m_i}\bar{\nu}_{ei}$ 
\\This leads to the following equations, where now the variables are non-dimensional
\begin{align*}
\partial_t n_k + \nabla_x \cdot(n_k u_k)=0, \quad k=i,e,
\end{align*}
\begin{align*}
\begin{split}
 \partial_t( n_i u_i)+ C_1 ~\nabla_x (n_i T_i) + \nabla_x \cdot ( n_i u_i \otimes u_i  ) - C_2 ~n_i (E+u_i \times B) = \\
C_3 ~\nu_{ie} n_e n_i (u_e - u_i),
\end{split}
\end{align*}
\begin{align*}
\begin{split}
C_4 ~ \partial_t( n_e u_e)+C_1~\nabla_x (n_e T_e) + C_4~ \nabla_x \cdot ( n_e u_e \otimes u_e  ) + C_2 ~ n_e (E+u_e \times B) = \\ C_3~ \nu_{ie}  n_e n_i (u_i - u_e),
\end{split}
\end{align*}
\begin{align*}
\begin{split}
C_1~\partial_t (\frac{3}{2} n_i T_i)+ \partial_t (\frac{1}{2} n_i |u_i|^2) +C_1~ \nabla_x \cdot(\frac{5}{2} n_i T_i u_i) + \nabla_x \cdot (\frac{1}{2} n_i |u_i|^2 u_i)\\ = C_2~ E  n_i u_i+C_3~ \frac{1}{2}\nu_{ie} n_e n_i ( |u_e|^2 - |u_i|^2)+\frac{1}{1+C_4} C_3 C_1 \nu_{ie}\frac{3}{2} n_i n_e  ( T_e -T_i),
\end{split}
\end{align*}
\begin{align*}
\begin{split}
C_1~\partial_t (\frac{3}{2} n_e T_e)+C_4~ \partial_t ( \frac{1}{2} n_e |u_e|^2) + C_1~\nabla_x \cdot(\frac{5}{2} n_e T_e u_e) +C_4 \nabla_x \cdot(\frac{1}{2} n_e |u_e|^2 u_e)\\ =-C_2~ E   n_e u_e+C_3~\frac{1}{2}\nu_{ei} n_e n_i( |u_i|^2 - |u_e|^2)+\frac{1}{1+C_4}C_3 C_1~ \nu_{ie}\frac{3}{2} n_i n_e  ( T_i -T_e),
\end{split}
\end{align*}
together with the Maxwell equations 
\begin{align*}
C_5 M ~\nabla_x \cdot E = \rho_c,
\end{align*}
\begin{align*}
\nabla_x \times E + \frac{\partial B}{\partial t}=0,
\end{align*}
\begin{align*}
\nabla_x \times B = j + M \frac{\partial E}{\partial t},
\end{align*}
\begin{align*}
\nabla_x \cdot B =0,
\end{align*}
\begin{align*}
\rho_c = (n_i-n_e),
\end{align*}
\begin{align*}
C_5 ~ j=(n_i u_i-n_e u_e).
\end{align*}
The constants $C_i, i=1,...,5$ and $M$ are dimensionless parameters. In particular, {\small $$C_1= \frac{\bar{n} \bar{T}}{m_i \bar{n} \bar{u}^2},\quad C_2= \frac{e \bar{B} \bar{t}}{m_i}, \quad C_3= \bar{\nu}_{ie} \bar{n} \bar{t}, \quad C_4= \frac{m_e}{m_i}, \quad C_5= \frac{\bar{j}}{e \bar{n} \bar{u}} \quad \text{and} \quad M=\frac{\bar{u}^2}{c^2},$$ } coming from non-dimensionalizing. The physical meaning is the following: $C_1$ describes the ratio over the typical scale of thermal energy $\bar{n} \bar{T}$ and of the kinetic energy $m_i \bar{n} \bar{u}^2$ of ions. If we consider an ion travelling with a speed perpendicular to a magnetic field at distance $r$, the force due to the magnetic field $e\bar{B} \bar{u}$ on the particle acts as a centripetal force $\frac{m \bar{u}^2}{r}$, so the norm of the forces is equal $$\frac{m \bar{u}^2}{r}= e\bar{B} \bar{u},$$
which is equivalent to $\omega:= \frac{\bar{u}}{r}= \frac{e \bar{B} \bar{u}}{m}$ which describes a frequency called cyclotron frequency. So $C_2$ is the product of the typical scale of the cyclotron frequency and the typical time scale. $C_3$ is the ratio of the macroscopic time scale and the time scale induced by the Mach number. $C_4$ is the mass ratio and $M$ the typical scale of the speed squared and the speed of light squared. Finally, $C_5$ is the typical scale of the current density induced by electric fields over the typical scale of the current induced by the flow of the particles.
\subsection{The limits to the MHD equations}
Now we consider the formal limit of the mass ratio $C_4 \rightarrow 0$ and the non-relativistic limit $M\rightarrow 0$.
\begin{theorem}
The formal limit of the mass ratio $C_4 \rightarrow 0$ and the non-relativistic limit $M\rightarrow 0$ of the system non-dimensionalized system with the remaining parameters remain finite is the system 
\begin{align*}
\partial_t n + \nabla_x \cdot (nu) = 0,
\\
 \partial_t( n u)+C_1 ~\nabla_x (n T) + \nabla_x \cdot (n u \otimes u  ) = C_2 C_5 ~ j \times B, 
\end{align*} 
\begin{align*}
\begin{split}
C_1~ \partial_t (\frac{3}{2}n T)+ \partial_t (\frac{1}{2} n |u|^2) + C_1 ~ \nabla_x \cdot ( \frac{5}{2} n u (T_e - T_i)) - C_1 C_5 \nabla_x \cdot (\frac{5}{2} T j)\\ + \nabla_x \cdot(\frac{1}{2} n |u|^2 u) =C_2 C_5 ~ E   j,
\end{split}
\end{align*}
\begin{align*}
\frac{C_3}{C_2} C_1 ~ \nabla_x (n T_e)  +  C_3 ~ n (E+u \times B)- C_3 C_5 (j \times B) = \frac{C_3^2}{C_2}  C_5 ~ \nu_{ei}  n j,
\end{align*}
\begin{align*}
\begin{split}
\frac{C_3}{C_2} C_1 ~ \partial_t (\frac{3}{2} n T_e)+ \frac{C_3}{C_2} C_1 ~ \nabla_x \cdot(\frac{5}{2} n T_e u) - \frac{C_3}{C_2} C_1 C_5 \nabla_x \cdot (\frac{5}{2} T_e j) \\ =- C_3 ~ E   n (u - C_5 \frac{j}{n})+ \frac{C_3^2}{C_2}  C_5 ~\nu_{ei} n j u - \frac{C_3^2}{C_2}  C_5^2 \frac{1}{2} \nu_{ei} |j|^2 \\+ \frac{C_3}{C_2} C_3 C_1 ~ \nu_{ei}\frac{3}{2} n^2  ( T_i -T_e),
\end{split}
\end{align*}
\begin{align*}
\nabla_x \times E + \frac{\partial B}{\partial t}=0,
\end{align*}
\begin{align*}
\nabla_x \cdot B =0,
\end{align*}
\begin{align*}
\nabla_x \times B =j,
\end{align*}
\begin{align*}
C_5 j=n (u- u_e).
\end{align*}
\label{Theorem1}
\end{theorem}
For the interested reader the proof is given in the Appendix. \\
Next, we consider the formal limit $C_5 \rightarrow 0$ and $\frac{C_3}{C_2} \rightarrow 0$, such that $C_2 C_5$ and $\frac{C_3^2 C_5}{C_2}$ remain bounded away from zero. Physically the first limit means that the current from moving particles $e \bar{n} \bar{u}$ dominates over the current due to electric forces $\bar{j}$. The second limit means that the cyclotron frequency $\frac{e \bar{B}}{m_i},$ dominates over the collision frequency $\bar{\nu}_{ie} \bar{n}$, while the current due to electric fields $\bar{j}$ per cyclotron time $1/\frac{e \bar{B}}{m_i}$ over the current induced by the flow $e \bar{n} \bar{u}$ in a typical time scale $\bar{t},$ remains bounded away from zero. Moreover, the ratio of the collision frequency and the cyclotron frequency is assumed to be of the same order of the electric current per collision time $\frac{1}{\bar{\nu}_{ie} \bar{n}}$ over the current induced by the flow per typical time scale.  All in all, we get the following theorem
\begin{theorem}
As $C_5 \rightarrow 0$ and $\frac{C_3}{C_2} \rightarrow 0$, such that $C_2 C_5$ and $\frac{C_3^2 C_5}{C_2}$ remain bounded away from zero, formally the solution of the system in Theorem \ref{Theorem1} tends to the solution of
\begin{align}
\partial_t n + \nabla_x \cdot (nu) &= 0,
\\
 \partial_t( n u)+ C_1 \nabla_x  (n T) + \nabla_x \cdot (n u \otimes u  ) &= C_2 C_5 ~  j \times B,
 \label{1}
\\
C_1 \partial_t (\frac{3}{2} n T)+ \partial_t (\frac{1}{2} n |u|^2) + C_1 \nabla_x \cdot(\frac{5}{2} n T u) + \nabla_x \cdot( \frac{1}{2} n |u|^2 u)& = C_2 C_5 ~ E   j,
\label{2}
\\
 (E+u \times B) &= \frac{C_3 C_5}{C_2} \nu_{ei} j,
 \label{c31}
 \\
 E u &=\frac{C_3^2 C_5}{C_2}\nu_{ei} j u,
 \label{c32}
\\
\nabla_x \times B &=  j,
\label{+}
\\
\nabla_x \times E + \frac{\partial B}{\partial t} &=0,
\label{3}
\\
\nabla_x \cdot B &=0 .
\end{align}
\end{theorem}
This is a direct consequence of Theorem \ref{Theorem1}.
Last we consider the formal limit $C_3 \rightarrow 0$ which means that interactions of ions and electrons can be neglected. In addition, we choose the special regime where $C_1=1$, that is $\bar{n} \bar{T}= m_i \bar{n} \bar{u}^2$ and $C_2 C_5=1$ in order to obtain the well-known conservation form for ideal MHD.
\begin{theorem}
As $C_3 \rightarrow 0$ and in the special regime $C_1=1$ and $C_2 C_5=1$, formally, we obtain the system of ideal MHD equations
\begin{align*}
\partial_t n + \nabla_x \cdot (nu) &= 0,
\\  \partial_t( n u)+ \nabla_x (u \otimes u n+( p +\frac{1}{2} |B|^2) \bf{1} - B \otimes B) &= 0,
\\
\partial_t (\frac{1}{2} n |u|^2 +\frac{3}{2}p+ \frac{1}{2} |B|^2) + \nabla_x \cdot( \frac{1}{2} n |u|^2 u+\frac{5}{2} p u + |B|^2 u)- B \cdot (B \otimes u))&= 0,
\\
\frac{\partial B}{\partial t} + \nabla_x \cdot (B \otimes u - u \otimes B)&=0,
\\
\nabla_x \cdot B &=0.
\end{align*} 
\label{Theorem3}
\end{theorem}
Again, for the interested reader the proof is given in the appendix.
\section{Conclusion and perspectives}
We derived a BGK equation for mixtures that replaces the Boltzmann collision operator satisfying the conservation properties, the H-theorem, positivity of solutions of positive initial data and positivity of all temperatures.
The BGK collision operator contains a sum of relaxation terms corresponding to each type of interaction, interaction with each species with itself and interaction with the other species. It has the advantage to single out the influence of a certain type of collision directly. For example, if one species has already reached a Maxwell distribution and the other one not.

First, we expect to extend the micro-macro decomposition \cite{Crestetto_2012} to gas mixtures.

Further work will be concentrated on the  problem of deriving the Navier-Stokes equations for mixtures from  kinetic equations as in \cite{Groppi, Brull_2012} in order to estimate accurate values for Fick's diffusion coefficient, the viscosity coefficient, the thermal conductivity and the thermal diffusion parameter. This model offers a possibility of matching experimental data because of the remaining free parameters $\alpha$, $\delta$ and $\gamma$.

\section*{Appendix}
 \begin{proof}[Proof of Theorem 4.1]
We start with the non-dimensionalized system from section 4.3.
In the limit $M \rightarrow 0$, we get from the first Maxwell equation that $n_i$ and $n_e$ converge formally to the same limit $n$. The third Maxwell equation simplifies to 
\begin{align*}
\nabla_x \times B = j .
\end{align*}
We denote the limit of $u_i$ by $u$. Then we get from conservation of the number of ions
\begin{align*}
\partial_t n + \nabla_x \cdot (nu) = 0.
\end{align*}
The momentum equation of the electrons turns into
\begin{align}
C_1 ~ \nabla_x (n T_e) + C_2 ~  n (E+u_e \times B) = C_3 ~\nu_{ie}  n n (u - u_e).
\label{mom_el}
\end{align}
The limit of the sum of the momentum equations with $T:=T_i+T_e$ gives
\begin{align}
 \partial_t( n u)+C_1 ~ \nabla_x (n T) + \nabla_x \cdot ( u \otimes u n ) + C_2 ~ n (u_e-u) \times B =0.
 \label{mom_sum}
\end{align} 
The other Maxwell equations turn into
\begin{align*}
\nabla_x \times E + \frac{\partial B}{\partial t}=0,
\end{align*}
\begin{align*}
\nabla_x \cdot B =0,
\end{align*}
\begin{align}
C_5 j=n (u- u_e).
\label{j}
\end{align}
The energy equation of the electrons leads to
\begin{align}
\begin{split}
C_1 ~\partial_t (\frac{3}{2} n T_e)+ C_1 ~ \nabla_x \cdot(\frac{5}{2} n T_e u_e) \\ =-C_2 ~ E   n_e u_e+C_3 ~ \frac{1}{2}\nu_{ei} n^2( |u|^2 - |u_e|^2)+C_3 C_1 ~ \nu_{ei}\frac{3}{2} n^2  ( T_i -T_e).
\label{en_el}
\end{split}
\end{align}
From the sum of the energy equations we get
\begin{align}
\begin{split}
C_1 ~ \partial_t (\frac{3}{2} nT)+\partial_t( \frac{1}{2} n |u|^2) + C_1 ~ \nabla_x \cdot(\frac{5}{2} ( n T_e u_e + n T_i u)) +\nabla_x \cdot( \frac{1}{2} n |u|^2 u)\\ =C_2 ~ E   n (u-u_e).
\label{en_sum}
\end{split}
\end{align}
Using $C_5 j = n(u-u_e)$, we get from \eqref{mom_el}, \eqref{mom_sum} and \eqref{en_sum} 
\begin{align}
C_1 ~ \nabla_x (n T_e)  + C_2 ~ n (E+u_e \times B) = C_3 C_5 ~ \nu_{ei}  n j,
\label{eq1}
\end{align}
\begin{align}
 \partial_t( n u)+C_1 ~\nabla_x (n T) + \nabla_x \cdot (n u \otimes u  ) = C_2 C_5 ~ j \times B, 
\end{align} 
\begin{align}
\begin{split}
C_1~ \partial_t (\frac{3}{2}n T)+ \partial_t (\frac{1}{2} n |u|^2) + C_1~\nabla_x \cdot(\frac{5}{2} ( n T_e u_e + n T_i u)) + \nabla_x \cdot(\frac{1}{2} n |u|^2 u)\\ =C_2 C_5 ~ E   j .
\end{split}
\label{eqj1}
\end{align}
Writing $|u|^2-|u_e|^2$ as $(u-u_e)\cdot (u+u_e)$ and again replacing $j$ by $C_5 j = n(u-u_e)$, we obtain form \eqref{en_el}
\begin{align}
\begin{split}
C_1 ~ \partial_t (\frac{3}{2} n T_e)+ C_1 ~ \nabla_x \cdot(\frac{5}{2} n T_e u_e) \\ =-C_2 ~ E   n_e u_e+C_3 C_5 ~\frac{1}{2}\nu_{ei} n j (u+u_e)+C_3 C_1 ~ \nu_{ei}\frac{3}{2} n n  ( T_i -T_e) .
\end{split}
\label{eq2}
\end{align}
Equations \eqref{eq1} and \eqref{eq2} are equivalent to
\begin{align}
C_1 ~ \nabla_x (n T_e)  + \frac{C_2}{C_3} C_3 ~ n (E+u_e \times B) = \frac{C_3^2}{C_2} \frac{1}{\frac{C_3}{C_2}} C_5 ~ \nu_{ei}  n j,
\label{eq3}
\end{align}
\begin{align}
\begin{split}
C_1 ~ \partial_t (\frac{3}{2} n T_e)+ C_1 ~ \nabla_x \cdot(\frac{5}{2} n T_e u_e) \\ =-\frac{C_2}{C_3} C_3 ~ E   n u_e+ \frac{C_3^2}{C_2} \frac{1}{\frac{C_3}{C_2}} C_5 ~\frac{1}{2}\nu_{ei} n j (u+u_e)+C_3 C_1 ~ \nu_{ei}\frac{3}{2} n^2  ( T_i -T_e) .
\end{split}
\label{eq4}
\end{align}
We multiply \eqref{eq3} and \eqref{eq4} by $\frac{C_3}{C_2}$ and insert $u_e= u- C_5 \frac{j}{n}$ from \eqref{j}, we get from \eqref{eqj1}, \eqref{eq3} and \eqref{eq4}
\begin{align}
\begin{split}
C_1~ \partial_t (\frac{3}{2}n T)+ \partial_t (\frac{1}{2} n |u|^2) + C_1 ~ \nabla_x \cdot ( \frac{5}{2} n u (T_e - T_i)) - C_1 C_5 \nabla_x \cdot (\frac{5}{2} T j)\\ + \nabla_x \cdot(\frac{1}{2} n |u|^2 u) =C_2 C_5 ~ E   j,
\end{split}
\end{align}
\begin{align}
\frac{C_3}{C_2} C_1 ~ \nabla_x (n T_e)  +  C_3 ~ n (E+u \times B)- C_3 C_5 (j \times B) = \frac{C_3^2}{C_2}  C_5 ~ \nu_{ei}  n j,
\end{align}
\begin{align}
\begin{split}
\frac{C_3}{C_2} C_1 ~ \partial_t (\frac{3}{2} n T_e)+ \frac{C_3}{C_2} C_1 ~ \nabla_x \cdot(\frac{5}{2} n T_e u) - \frac{C_3}{C_2} C_1 C_5 \nabla_x \cdot (\frac{5}{2} T_e j) \\ =- C_3 ~ E   n (u - C_5 \frac{j}{n})+ \frac{C_3^2}{C_2}  C_5 ~\nu_{ei} n j u - \frac{C_3^2}{C_2}  C_5^2 \frac{1}{2} \nu_{ei} |j|^2 + \frac{C_3}{C_2} C_3 C_1 ~ \nu_{ei}\frac{3}{2} n^2  ( T_i -T_e) .
\end{split}
\end{align}
\end{proof}
\begin{proof}[Proof of Theorem \ref{Theorem3}]
In the limit $C_3 \rightarrow 0$, the equations \eqref{c31} and \eqref{c32} turn into
\begin{align}
 E+u \times B &= 0,
 \label{-}
 \\
 E u &= 0 .
\end{align}
 We insert $j= \nabla_x \times B$ from \eqref{+} into \eqref{1}. The j-th component of the term $(\nabla_x \times B)\times B$ can be simplified to $\sum_{n=1}^3 B_n ( \partial_{x_j} B_n - \partial_{x_n} B_j), \quad j=1,2,3$. Since $\nabla_x \cdot B=0$, we can add $\nabla_x \cdot B B_j,$ so we get $\sum_{n=1}^3 B_n ( \partial_{x_j} B_n - \partial_{x_n} B_j) + \nabla_x B B_j$ which is the $j-$th component of $- \nabla_x \cdot ( \frac{1}{2} |B|^2 \bf{1} - B \otimes B).$ Thus, \eqref{1} turns into 
$$  \partial_t( n u)+ \nabla_x (u \otimes u n+( n T +\frac{1}{2} |B|^2) \bf{1} - B \otimes B) = 0. $$
Now, we insert $E= - u \times B$ from \eqref{-} into \eqref{3}.  In a similar way again using $\nabla_x B = 0$, we obtain 
$ - \nabla_x \times (u \times B)= \nabla_x \cdot (B \otimes u - u \otimes B),$ so \eqref{3} leads to
\begin{align}
\frac{\partial B}{\partial t} + \nabla_x \cdot (B \otimes u - u \otimes B)=0.
\label{neu}
\end{align}
Finally, inserting \eqref{+}, \eqref{-}, $\nabla_x B=0$ and \eqref{neu} into \eqref{2}, leads to
$$ \partial_t (\frac{1}{2} n |u|^2 +\frac{3}{2} n T+ \frac{1}{2} |B|^2) + \nabla_x \cdot( \frac{1}{2} n |u|^2 u+\frac{5}{2} n T u + |B|^2 u)- B \cdot (B \otimes u))= 0.$$
\end{proof}

\end{document}